\documentclass{amsart}   %%%%%%%Please do not change
\usepackage{graphicx, amsmath, amssymb, amsthm}
\usepackage{color}
\usepackage{epsf}
\usepackage{enumerate}
\newtheorem{thm}{Theorem}[section]
\newtheorem{cor}[thm]{Corollary}

\newtheorem{claim}[thm]{Claim}

\newtheorem{que}{Question}[section]

\theoremstyle{definition}

\newtheorem{rem}[thm]{Remark}
\newtheorem{exmp}[thm]{Example}

\numberwithin{equation}{section}

\DeclareMathOperator{\diam}{diam}

\newcommand{\Z}{\mathbb{Z}}
\newcommand{\R}{\mathbb{R}}

\newcommand{\eps}{\varepsilon}

\begin{document}

\vspace{0.5in}

\title[A compact minimal space $Y$ such that its square $Y\times Y$ is not minimal]{A compact minimal space $Y$ such that its square $Y\times Y$ is not minimal}
%[On entropy of pseudocircle self-maps with rotational chaos]%
%{ \\Example For The Authors}

\author{J. P. Boro\'nski}
\address[J. P. Boro\'nski]{National Supercomputing Centre IT4Innovations, Division of the University of Ostrava,
Institute for Research and Applications of Fuzzy Modeling,
30. dubna 22, 70103 Ostrava,
Czech Republic -- and --
AGH University of Science and Technology, Faculty of Applied Mathematics,
al. Mickiewicza 30,
30-059 Krak\'ow,
Poland}
\email{jan.boronski@osu.cz}

\author{Alex Clark}
\address[A. Clark]{Department of Mathematics, University of Leicester, University Road, Leicester LE1 7RH, United Kingdom}
\email{Alex.Clark@le.ac.uk}

\author{P. Oprocha}
\address[P. Oprocha]{AGH University of Science and Technology, Faculty of Applied
Mathematics, al.
Mickiewicza 30, 30-059 Krak\'ow, Poland
-- and --
National Supercomputing Centre IT4Innovations, Division of the University of Ostrava,
Institute for Research and Applications of Fuzzy Modeling,
30. dubna 22, 70103 Ostrava,
Czech Republic}
\email{oprocha@agh.edu.pl}

\begin{abstract}
The following well known open problem is answered in the negative: Given two compact spaces $X$ and $Y$ that admit minimal homeomorphisms, must the Cartesian product $X\times Y$ admit a minimal homeomorphism as well? Moreover, it is shown that such spaces can be realized as minimal sets of torus homeomorphisms homotopic to the identity. A key element of our construction is an inverse limit approach inspired by combination of a technique of Aarts \& Oversteegen and the construction of Slovak spaces by Downarowicz \& Snoha \& Tywoniuk. This approach allows us also to prove the following result. Let $\phi\colon M\times\mathbb{R}\to M$ be a continuous, aperiodic minimal flow on the compact, finite--dimensional metric space $M$. Then there is a generic choice of parameters $c\in\mathbb{R}$, such that the homeomorphism $h(x)=\phi(x,c)$ admits a noninvertible minimal map $f\colon M\to M$ as an almost 1-1 extension. %We also show a family of spaces, with the property announced in the tittle, each of which is a minimal set of a torus homeomorphism homotopic to the identity. 
\end{abstract}

\maketitle
\section{Introduction}
\subsection{Motivation}
Given a compact metric space $X$, a map $f:X\to X$ is said to be \emph{minimal} if the forward orbit $\{f^n(x):n=1,2,\ldots\}$ is dense in $X$, for every $x\in X$. In such a case we call $X$ minimal with respect to $f$, or simply minimal if there is no confusion as to the map. In the present paper we contribute to the following two well known problems. %We shall refer to a space $X$ as minimal if it admits a minimal map.
\begin{itemize}
	\item[(Q1)] Is minimality with respect to homeomorphisms preserved under Cartesian product in the class of compact spaces?
	\item[(Q2)] What spaces admit minimal noninvertible maps?
\end{itemize}
\subsection{Minimality of products of minimal spaces}
As with many other topological and dynamical properties, a fundamental question is to determine whether the property of being a minimal space is preserved under the Cartesian product, and it seems quite surprising that it has not been settled until now. Note that, for example for the fixed point property this question had been resolved in the negative already 50 years ago by Lopez \cite{Lopez}. In addition, basic examples of minimal spaces such as the circle, or the Cantor set, are known to preserve minimality under the Cartesian product, and there have not been even prospective candidates to provide a counterexample. Although no product of a homeomorphism with itself $(h,h):X\times X\to X\times X$ is minimal, by the fact that it keeps the diagonal $\Delta=\{(x,x)\,|\,x\in X\}$ invariant, the Cartesian product typically gives rise to new homeomorphisms that do not factor as conjugate homeomorphisms onto both coordinate spaces, and in this way minimal homeomorphisms of the product can often be obtained. In the present paper, however, we finally settle this problem in the negative, by providing a class of counterexamples. Each space $Y$ in the class is such that it admits a minimal homeomorphism but no Cartesian power of it $Y^n$ does. Each such space also admits a monotone map onto a suspension of a minimal Cantor system, and the minimal homeomorphisms it admits are extensions of minimal homeomorphisms of the suspension. Note that a similar counterexample does not exist for flows. In that case minimality is preserved by countable Cartesian products, as shown by Dirb\'ak \cite{Dirbak}.
\subsection{Minimal noninvertible maps}
In the late 1960s Auslander raised the question concerning the existence of minimal noninvertible maps. In a joint work with Yorke he showed that the Cantor set admits such maps \cite{AY}. Since then more examples have been given, but as in the case of minimal homeomorphisms and flows, the classification of spaces admitting such maps is a difficult well known open question. In 1979 Auslander and Katznelson showed that the circle admits no such maps \cite{AK}, despite supporting minimal homeomorphisms. In 2001 Kolyada, Snoha and Trofimchuk constructed such maps for the 2-torus, proving that any minimal skew product homeomorphism of the 2-torus $\mathbb{T}^2$ having an asymptotic pair of points has an almost one-to-one factor
which is a noninvertible minimal map of $\mathbb{T}^2$ \cite{Kolyada}. Modifying this approach, in 2003 Bruin, Kolyada and Snoha proved in \cite{Bruin} that any minimal skew product homeomorphism of $\mathbb{T}^2$ having an asymptotic pair of points has a factor which is a noninvertible minimal map of a 2-dimensional nonhomogeneous metric continuum $X$, such that any homeomorphism of $X$ has a fixed point. Tywoniuk showed in \cite{Tywoniuk} that solenoids, unlike the circle, admit noninvertible minimal maps. It had been a long-standing open question if the circle is the only nondegenerate continuum that admits a minimal homeomorphism but no minimal noninvertible map. It is a well known unresolved conjecture that the pseudo-circle is another such continuum, motivated by the fact that it admits minimal homeomorphisms \cite{Handel}. Recently Downarowicz, Snoha and Tywoniuk \cite{Downarowicz} have answered the general question in the negative by providing a family of counterexamples that belong to the class of so-called Slovak spaces. These spaces have their homeomorphism groups generated by a minimal homeomorphism, which in some cases has positive entropy, but they all admit no minimal noninvertible maps. In the present paper we show a new, very general class of compact spaces that admit minimal noninvertible maps. Namely, any compact, finite--dimensional metric space that admits a continuous, aperiodic minimal flow belongs to that class. In addition to some manifolds with zero Euler characteristic, examples include peculiar minimal sets of flows on compact manifolds, such as the Denjoy and Kuperberg minimal sets.
\subsection{Structure of the paper}
We shall start in Section 2 with the construction of a new class of spaces admitting minimal noninvertible maps. The maps are obtained as perturbations of time $t$ homeomorphisms of aperiodic minimal continuous flows on metric spaces. Our inverse limit approach is motivated by the method introduced by Aarts \& Oversteegen, who in \cite{Aarts} used it to construct a transitive homeomorphism of the Sierpinski Carpet.
A novel variant of this method is then used in Section 3 to modify minimal Cantor systems suspensions to obtain minimal continua, none of which admits minimality in the Cartesian product of finitely many copies of itself. These continua resemble some of the Slovak spaces constructed by different methods by Downarowicz, Snoha and Tywoniuk, where a dense orbit of a minimal homeomorphism was blown up to a null seqeunce of $\sin(1/x)$-curves. Our inverse limit approach allows us to introduce a null sequence of pseudo-arcs instead, which in the resulting spaces forces factorwise rigidity, as well as almost cyclicity of their homeomorphism groups, in the sense that they are isomorphic to either $\mathbb{Z}$ or $\mathbb{Z}\otimes\mathbb{Z}_2$. In the final section of the paper we illustrate how these techniques can be applied in a wider context to spaces that do not admit minimal flows but instead minimal maps that preserve a local product structure in a sufficiently smooth manner.
\section{Minimal noninvertible maps}
In what follows, we identify the unit circle with $S^1=\R/\Z$ and the 2-dimensional torus with $\mathbb{T}^2=S^1\times S^1$.

A continuous surjection $\pi \colon X\to Y$ is \textit{almost 1-1} if the set
$$
Y_1=\bigcap_{n=1}^\infty \left\{y : \diam (\pi^{-1}(y))<1/n\right\}
$$
is dense in $Y$. A dynamical system $(X,f)$ is an \textit{almost 1-1 extension} of $(Y,g)$
if there exists an almost 1-1 factor $\pi \colon (X,f)\to (Y,g)$. Note that if $(X,f)$ is a minimal almost 1-1
extension of $(Y,g)$ then $\pi^{-1}(Y_1)$ is dense in $X$. In fact both sets $Y_1$ and $\pi^{-1}(Y_1)$ are residual in that case.

In \cite{Kolyada} Kolyada, Snoha, and Trofimchuk proved that any minimal skew product homeomorphism of the 2-torus having an asymptotic pair of points has an almost 1-1 factor which is a noninvertible minimal map. By an inverse limit technique of Aarts and Oversteegen \cite{Aarts}, as well as the application of Lemma 3.1 in \cite{Kolyada}, we prove the following related result.

\begin{thm}\label{generalnonivert}
Let $\phi\colon M\times\mathbb{R}\to M$ be a continuous, aperiodic minimal flow on the compact, finite--dimensional metric space $M$. Then there is a generic choice of parameters $c\in\mathbb{R}$, such that the homeomorphism $h(x)=\phi(x,c)$ admits a noninvertible minimal map $f\colon M\to M$ as an almost 1-1 extension.
\end{thm}

\begin{proof}
By standard techniques, see e.g. \cite{Fayad}, there exists a residual set $A$ of parameters $t$ such that the time--$t$ map $F_t:=\phi(t,\cdot)$ is minimal, and hence we may fix such a $t_0>0$ and denote $F=F_{t_0}.$  By the theorem of Bebutov and its consequences (see, e.g., \cite[pp. 333--338]{NS}) there is for each $x \in M$ and $T\in \R,$ a local section $ S(x,T)$ with corresponding flow box neighbourhood $U(x,T)$ and a homeomorphism onto its image $h\colon U(x,T) \to \R^{d}$ that conjugates the time $t$-tmap $\phi(\cdot,t)$    of the flow restricted to $U(x,T)$ with addition by $t$ in the first coordinate on the image. Specifically,  for some $\delta >0$ the $\delta$ ball centered at $x$, $B(x,\delta)$  contains a set $ S(x,T)$ so that for each $y\in \bigcup_{t \in [-T,T]} F_t\left(B(x,\delta)\right)=U(x,T)$ there is a unique $t_y$ such that $\phi(-t_y,y)\in S(x,T),$ and the homeomorphism $h$ maps $S(x,T)$ into the set $ \left\{(y_i)\in \R^d \, | \, y_1=0\right\}$ and the point $y$ as above to $(t_y,y_2,\dots,y_d),$ where $(0,y_2,\dots,y_n)=h\left(\phi(-t_y,y)\right).$  The fact that we can choose a finite--dimensional Euclidean space $\R^{d}$ follows from the standard fact that finite--dimensional separable  metric spaces can be embedded in a finite--dimensional Euclidean space. We now fix a particular $x_0\in M$ and then construct for $n\in \Z^+$ a nested sequence of local sections $ S(x_0,nt_0)$  (i.e. $S(x_0,(n+1)t_0)\subset S(x_0,nt_0)$) and corresponding flow box neighbourhoods $U(x_0,nt_0)$, with corresponding homeomorphisms $h_n\colon U(x_0,nt_0) \to \R^d$ onto their images, constructed in such a way that $h_{n}=h_{n+1}$ when restricted to $U(x_0,nt_0)\cap U(x_0,(n+1)t_0)$ and $h_n(x_0)=0$ for all $n.$

%Now choose a Lebesgue number $\eps>0$ of the cover $\left\{ U_i \right\} $ and choose $T>0$ so that for each  every $x$ we have $\phi(x,[0,T])\subset %B(x,\eps/2)$. In particular, each arc $\phi(x,[0,T])$ (together with a small flow box neighborhood) is always contained in some $U_i$. By \cite{Fayad} there %exists a residual set $A$ of parameters $t$ such that the time-$t$ map $F_t=\phi(t,\cdot)$ is minimal, and so we fix such a $0<t<T$ and denote $F=F_t$.

For each $n\in \Z$ denote $x_{n}=F ^n(x_0)$. We put $X_0=M$ and will define spaces $X_n$ for $n>0$ inductively.
Using $\| x- y \|$ to denote the Euclidean metric in $\R^d$, we define a new metric $D$ on $\R^d\setminus\{0\}$ by
  \[
  D(x,y)=\|x-y\|+|c(x)-c(y)|, \text{ where } c(y)=\frac{y_1}{\sqrt{\sum_{i=1}^{d} y_i^2}}\;.
  \]
The completion of $\R^d\setminus\{0\}$ with respect to $D$ has remainder given by an interval that can naturally be identified with $[-1,1]$, corresponding to the limiting value of $c$ for the points in the remainder. By the compactness of $[-1,1],$ any restriction of this completion to a compact subset of $\R^d$ intersected with $\R^d\setminus\{0\}$ is compact. The level sets of $c$, $c^{-1}(x),\,\,\,x\in (-1,1)$ form cones that separate $\R^d\setminus\{0\}$ and intersect \emph{each} line that is an image of a flow line in exactly one point. Since the level sets of $c$ are closed under non--zero scalar multiplication, any neighbourhood of $0$ in $\R^d$ will intersect each level set.  Thus, under the hypotheses of aperiodicity and minimality, the completion of the images of $h_n$ in $\R^d$ will also have intervals as remainders.

Denote by $X_1$ the compactification of $X_0\setminus \{x_{-1}\}$ by $I_1=[-1,1]$ in the following way. We take $h_1 \left( U(x_0,t_0)\right)+(t_0,0,\dots,0)$ so that the image of $x_{-1}$ is at $0$ and we complete using the metric $D$ as described above. We then identify $I_1$ with the interval obtained as remainder in the completion of the metric $D.$ In the same way, $X_2$ is obtained from $X_1$ by blowing up the point $x_{-2},$ and recursively the space $X_{n+1}$ is obtained from $X_n$ by removing the point $x_{-n}$ and compactifying the resulting hole by a closed interval $I_n$ identified with $[-1,1]$ as with $I_1.$
	It is not hard to see that $X_1$ is homeomorphic to $X_0$ because $\overline{U(x_0,t_0)}$ and the compactification of $\overline{U(x_0,t_0)}\setminus \{x_{-1}\}$ in $X_1$ are homeomorphic by a homeomorphism which is the identity on the boundary of $\overline{U(x_0,t_0)}$. By the same argument all $X_n$ are homeomorphic to $X_0=M$.
Observe that compactification of $X_0\setminus\{x_{-1},\ldots, x_{-n}\}$ by intervals using the above method gives raise to a space homeomorphic to $X_n$.
This is due to the fact that $h_{n+1}|_{U(x_0,nt_0)\cap U(x_0,(n+1)t_0)}=h_{n}|_{U(x_0,nt_0)\cap U(x_0,(n+1)t_0)}$.
	
Denote by $X_\infty$ the inverse limit of the spaces $X_n$ with bonding maps $\pi_n \colon X_{n+1}\to X_n$ defined by $\pi_n(x)=x$ for $x\not \in I_{n+1}$ and $\pi_n(x)=x_{-n-1}$ for $x\in I_{n+1}$. Then by a result of Brown \cite[Theorem 4]{Brown} we obtain that $X_\infty$ is homeomorphic to $M$ because all the maps $\pi_n$ are monotone.
	
	Now we shall define a minimal but noninvertible map $H$ on $X_\infty$. Fix any $\underline{x}\in X_\infty$. If for every $n\geq 0$ the coordinate $\underline{x}_{\;n}\not\in I_n$ then we put $H(\underline{x})_i=F(\underline{x}_i)=F(\underline{x}_0)$ for every $i\geq 0$. Now suppose that for some $n$ we have $\underline{x}_{\;n}\in I_n$. Then $\underline{x}_i\in I_n$ for all $i>n$
and $\underline{x}_i=x_{-n}$ for $i<n$. If $n=1$ then we put $H(\underline{x})_i=x_0$ for every $i\geq 0$. For the case $n>1$, observe that since both $I_n$ and $I_{n+1}$
are identified with $[-1,1]$, we have a homeomorphism $\psi \colon I_{n+1}\to I_n$ which is the identity after this identification. We put $H(\underline{x})_i=\psi(\underline{x}_i)$
for $i\geq n$ and for $i<n$ we put inductively $H(\underline{x})_i=\pi_{i}\left(\underline{x}_{i+1}\right)$. This way $H$ is defined on $X_\infty$ and it is clear that it is  surjective.
We must show that it is continuous. If $\underline{x}\in X_\infty$ is such that  $\underline{x}_{\;n}\not\in I_n$ for any $n>0$ then it is clear that $H$ is continuous at
$\underline{x}$ because in each $X_i$ there is a neighborhood of $\underline{x}_{i}$ disjoint from $I_1\cup \cdots \cup I_i$. On the other hand if for some $n$ $\underline{x}_{\;n} \in I_n$,  $x_{-n}$ and $x_{-n-1}$
are contained in a flow box $U(x_0,(n+1)t_0)$ that was used in the construction of the compactification, which means that $F$ can be locally identified with a translation in $\R^d$ that preserves the level sets of $c$ used in the definition of the compactifications. It easily follows that $H$ is continuous also at such $\underline{x}$.
	
Finally, we must show that $H$ is minimal. Every minimal homeomorphism is forward minimal, hence for every $z\in M$ its forward orbit $\{F^n(z): n\geq 0\}$ is dense in $M$. In particular, the forward orbit of every $z\in M\setminus\{x_{-n}: n\geq 1\}$ is dense in $M\setminus\{x_{-n}: n\geq 1\}$. Furthermore, for every $y\in M$ and every $a<b$, the segment $\phi(y,[a,b])$ has empty interior in $M$, because otherwise $M$ would be a single closed orbit.
	Then it is clear that for every $\underline{x}\in X_\infty$ and every $n\geq 0$ the set $\{H^i(\underline{x})_n : i\geq 0\}$ is dense in $X_n$. But then the forward orbit of $\underline{x}$ under $H$ is dense in $X_\infty$. This shows that $H$ is minimal. It is also clear that $H$ is not invertible.
	
By the same argument, the natural projection $\pi\colon X_\infty \to M$ is one-to-one on a dense set, hence is almost 1-1 extension (see also \cite[Theorem 2.7]{Kolyada}).
	The proof is completed.
\end{proof}
\begin{rem}\label{rem:r1}
Since our goal was to construct a noninvertible map, we have only compactified with intervals for points in the negative orbit of $x_0.$ However, it is clear that we could have compactified with intervals for points in the complete orbit of $x_0$ to form an invertible map instead.
\end{rem}

A large but not exhaustive class of spaces admitting a minimal flow can be obtained by suspension as follows. Let $h\colon C\to C$ be a minimal homeomorphism of a compact metric space $C.$ Put $X=C\times \R /_\sim$,  where $\sim$ is the equivalence relation given by:
$(x,y)\sim (p,q)$ provided that $y-q\in \Z$ and $p=h^{-y+q}(x)$.  Then the \textit{suspension flow} defined by $h$
is the continuous flow induced on $X$ given by $\phi_t(x,s)=(x,s+t)/_\sim$. Since the orbits of $h$ are dense, the flow orbits are dense in $X$, and so $X$ is a continuum.
As before, there exists a residual set $A$ of parameters $t$ such that $\phi_t$ is minimal. By a \textit{minimal suspension} of $h$ we mean any homeomorphism $h=\phi_t$ for $t\in A$. And as long as $C$ is infinite,  $\phi$ is aperiodic. Then a direct application of Theorem~\ref{generalnonivert} gives the following.

\begin{cor}\label{r2:t2}
Suppose that $\phi$ is suspension flow defined by a minimal homeomorphism $h\colon C\to C$ on an infinite, finite--dimensional compact metric space $C$ and let $X$ be the phase space of $\phi$. Then $X$ is a continuum that admits a noninvertible minimal map.
\end{cor}

In the following section we will make use of this construction with $C$  the Cantor set.
\begin{exmp}
Let $\phi$ be the suspension flow of an adding machine and $X$ be its phase space. Then $X$ is a solenoid. By Corollary~\ref{r2:t2} we obtain that $X$
admits a noninvertible dynamical system. Such a system was constructed before in \cite{Tywoniuk} by a completely different, long and technical argument.
\end{exmp}

\begin{exmp}
Let $\Sigma$ be a Kuperberg Minimal Set of the smooth flow $\phi$ on $\mathbb{S}^3$ without a closed orbit, first constructed in \cite{KK}.  The space $\Sigma$ is a continuum with unstable shape \cite{Hurder} that admits a noninvertible minimal map $f\colon \Sigma\to \Sigma$.
\end{exmp}

\begin{exmp}
Consider the standard Denjoy minimal set for circle homeomorphism $h_{\alpha}$ obtained by a blow-up of $n$ orbits in the irrational circle rotation by $\alpha$ (e.g. see \cite{Dev}). Clearly this minimal set is the unique invariant Cantor set $C$ for $h_{\alpha}$. Let $\phi_\alpha$ be the suspension flow of $h_{\alpha}|_C$  and $X(\alpha)$ be its phase space. Then $X(\alpha)$ belongs to the class of so-called Denjoy continua. This is a well-studied class of spaces, whose elements played a crucial role in Schweitzer's counterexample to the Seifert Conjecture \cite{Schweitzer}.  By Corollary~\ref{r2:t2} we obtain that each $X(\alpha)$ admits a noninvertible minimal dynamical system. Alternatively, we could realise these as minimal subsets of flows on the torus.
\end{exmp}

\begin{rem}
It is well known that the only compact surfaces admitting a minimal system are the 2-dimensional torus and the Klein bottle (e.g. see \cite[Theorem~3.17]{Blokh}). On the other hand, every flow on the Klein bottle has a closed orbit (e.g. see \cite{Markley}, cf. \cite{Knes}). Hence, $\mathbb{T}^2$ is the only compact surface admitting a minimal flow, and so this is the only closed surface to which Theorem~\ref{generalnonivert} applies. It is not clear (beyond tori and related examples) which higher--dimensional closed manifolds admit minimal flows -- even the case of $S^3$ remains unknown.
\end{rem}
\section{A minimal continuum $Y$ such that $Y^n$ is not minimal for all $n>1$}

In this section we adapt the technique of the previous section to replace the inserted arcs with pseudo--arcs to create a space with special properties. Recall that in \cite{Downarowicz} Downarowicz, Snoha and Tywoniuk introduced the notion of Slovak spaces. We say that a compact $X$ is a \textit{Slovak space} if its homeomorphism group $H(X)$ is cyclic and generated by a minimal homeomorphism. In the present section we shall appeal to a slightly broader class of spaces which we define as follows. We say that a compact space $X$ is an \textit{almost Slovak space} if its homeomorphism group $$\mathrm{H}(X)=\mathrm{H}_+(X)\cup \mathrm{H}_{-}(X),$$
with
$$\mathrm{H}_{+}(X)\cap \mathrm{H}_{-}(X)=\{\operatorname{id}_X\},$$
where $\mathrm{H}_+(X)$ is cyclic and generated by a minimal homeomorphism, and for every $g\in \mathrm{H}_{-}(X)$ we have $g^2\in \mathrm{H}_{+}(X)$.

Let $h\colon C\to C$ be a minimal homeomorphism of a Cantor set $C$.
%Put $X=C\times \R /_\sim$ where $\sim$ is an equivalence relation such that
%$(x,y)\sim (p,q)$ provided that $x-p\in \Z$ and $y=h^{-x+p}(q)$. It is clear that $X$ is a continuum, since $\varphi$ has a dense orbit and so some (in fact all) composant is dense in $X$.
%Then \textit{suspension flow} defined by $\varphi$
%is a continuous flow induced on $X$ by $\phi_t(x,s)=(x,s+t)/_\sim$.
%By \cite{Fayad} there exists a residual set $A$ of parameters $t$ such that $\phi_t$ is minimal. By a \textit{minimal suspension} of $\phi$ we mean any homeomorphism $h=\phi_t$ for $t\in A$.
%
%
Now we shall adapt the approach from the proof of Theorem \ref{generalnonivert}, to construct an almost Slovak space $Y$. We shall later show that with the appropriate choice of $h$ $Y^n$ is minimal only if $n=1$. 

Recall that the \emph{composant} of the point $x$ of the space $X$ is the union of all proper subcontinua of $X$ containing $x.$

\begin{thm}\label{thm:h}
Let $(C,h)$ be a minimal homeomorphism of a Cantor set $C$. There exists a minimal suspension $(X,F)$ of $(C,h)$, a continuum $Y$, a minimal homeomorphism $(Y,H)$ and a factor map $\pi \colon (Y,H)\to (X,F)$ such that:
\begin{enumerate}[(i)]
	\item $\pi$ is almost 1-1,
	\item all non-singleton fibers $\pi^{-1}(q)$ are pseudo-arcs,
	\item there exists a composant $W\subset Y$ such that if $|\pi^{-1}(x)|>1$ then $\pi^{-1}(x)\subset W$.
	\item $\lim_{|i|\to \infty}\diam H^i\left(\pi^{-1}(x)\right)=0$ for all $x$
\end{enumerate}
\end{thm}
\begin{proof}
We may assume that $C\subset [-1,1]$ and $(-\delta,0)\cap C\neq \emptyset$ and $(0,\delta)\cap C\neq \emptyset$
for every $\delta>0$.
By the results of \cite{Fayad} we may also assume that $F=\phi_t$ of the suspension flow of $h$ for some $t$  so that the construction from Theorem~\ref{generalnonivert} can be carried out where instead of replacing the points of only the backward orbit of a particular point $q$ by arcs, we replace all points in the orbit of $q$ with a null sequence of arcs (see Remark~\ref{rem:r1} and Corollary~\ref{r2:t2}).

More specifically, fix any $q\in X$. By Theorem~\ref{generalnonivert} there exists an almost 1-1 extension $(X,H)$ of $(X,F)$ and a sequence of maps $(X,\pi_n)$, $(X,\eta_n)$ (used to construct $H$ by an inverse limit technique) such that $\pi_n \colon X\to X$ is one-to-one except on the set $A_n=\{ F^i(q) : |i|<n\}$, where each $\pi_n^{-1}(y)$, $y\in A_n$ is an interval. Furthermore $\pi_{n+1}=\pi_n \circ \eta_n$, where $\eta_n$ is one-to-one at every point except on two intervals $\pi^{-1}_n(F^{-n}(q)),\pi^{-1}_n(F^{n}(q))$, each of which is collapsed to a distinct point. Note that for sufficiently small $\delta>0$ the map $\eta_n$ is invertible on the complement of the set $\pi_n^{-1}(B(A_n,\delta))$.

Let $g\colon [0,1]\to [0,1]$ be a map such that $g(0)=0$, $g(1)=1$
and the inverse limit of $[0,1]$ with $g$ as the sole bonding map is the pseudo-arc (e.g. see \cite{Henderson} or \cite{Minc}).
Extend $g$ to a continuous surjection on $[-1,2]$ by putting $g(x)=x$ for all $x\not\in (0,1)$.

Put $X_0=X$. Let $I=\pi_1^{-1}(q)$ and let $U$ be a sufficiently small neighborhood of $I$
homeomorphic to $(-\eps,1+\eps)\times C$, where $C$ is a Cantor set and $I\approx [0,1]\times \{0\}$.
There exists a nested sequence of clopen sets $C_k \subset C$ such that $\bigcap_k C_k=\{0\}$.

Now for each $n$ let $f_n^{(1)}\colon X\to X$ be a map such that $f_n^{(1)}|_{X\setminus U}=\text{id}$,
$f_n^{(1)}(x,y)=(x,y)$ for $(x,y)\in U$, $y\not\in C_n$, and $f_n^{(1)}(x,y)=(g(x),y)$ otherwise. Let
$X_1=\varprojlim ((X,f_n^{(1)})_{n=1}^\infty)$. Observe that for each $x\not\in I$ there exists an $N$, and an open neighborhood $V$ of $(x,y)$, such that $f_n^{(1)}(x',y)=(x',y)$ for all $(x',y)\in V$ and $n>N$. But since a finite number of the first coordinates in a sequences from an inverse limit does not affect the topological structure of $X_1$, we see that if $x\in X_1$ and $x_n\not \in I$ for every $n$, then a small neighborhood of $x$ is homeomorphic to $C\times (0,1)$. If, on the other hand, $x_n\in I$ for some (thus all) $n$
then $x\in \varprojlim ((I,f_n^{(1)})_{n=1}^\infty)=\varprojlim (I,g)$, which is a unique maximal pseudo-arc embedded in $X_1$. We may view $X_1$ as $X$ with removed $q$ and the resulting ``hole'' compactified by a pseudo-arc. Note that we have a natural projection $\xi_1\colon X_1\to X$ given by $(\xi_1(x))_n=\pi_1(x_n)$, where we identify $X=\varprojlim (X,\text{id})$.

By the same method we define $X_n$ and $\xi_n\colon X_n\to X$ which is one-to-one except on the set $A_n,$ where $\xi_n^{-1}(y)$ is a pseudo-arc for every $y\in A_n$. Furthermore observe that $\eta_n$
induces a natural projection $\gamma_n \colon X_{n+1}\to X_n$ which collapses each of two pseudo-arcs to a point.

Let $Y=\varprojlim ((X_n,\gamma_n)_{n=1}^\infty)$ and let $\pi \colon Y\to X$ be the natural projection induced by the maps $\xi_n$.
Observe that $\pi$ is one-to-one onto every point that does not belong to the orbit of $q$, and $\pi^{-1}(F^i(q))$ is a pseudo-arc for every $i\in\mathbb{Z}$. Furthermore, all composants of $Y$ are topological one-to-one images of the real line, except the composant $W$ containing countably many pseudo-arcs connected by arcs (note that in $\varprojlim ([-\delta,1+\delta],g)$
the constant sequences $x=(0,0,\ldots)$ and $x=(1,1,\ldots)$ belong both to the pseudo-arc $\varprojlim ([0,1],g)$ and the arcs $\varprojlim ([-\delta,0],g)$ and $\varprojlim ([1,1+\delta],g)$, respectively).

It remains to define a homeomorphism $H\colon Y\to Y$. For each $y\in Y$ such that $\pi(y)\not\in \{F^i(q) : i\in \Z\}$ we put $H(y)=\pi^{-1}\circ F \circ \pi(y)$.
If $\pi(y)=F^i(q)$ then take any $n>i$ and denote $I^n_i=\pi_n^{-1}(F^i(q))$ and $I^n_{i+1}=\pi_n^{-1}(F^{i+1}(q))$. By the method of our construction we have a linear map $H_n\colon I^n_i\to I^n_{i+1}$
that extends to a continuous map in a neighborhood of $I^n_i$ in such a way that $\pi_n\circ H_n\circ \pi_n^{-1}(y)=F(y)$ for all $y$ sufficiently close to $F^i(q)$.
Furthermore, for each $n>i$ we have $H_n \circ \gamma_n(x)=\gamma_n \circ H_n(x)$ for each $x\in I_i^n$. This way $H$ is defined also on $\pi^{-1}(F^i(q))$ for every $i\in \Z$, where it is a homeomorphism, since $H_n$
is a homeomorphism onto its image on the set $X_n\setminus \{\pi_n^{-1}(F^n(q)),\pi_n^{-1}(F^{-n-1}(q))\}$. Note that each pseudo-arc in $X_n$ has empty interior in $X_n$, hence the set of points $y$
such that $\pi^{-1}\pi(y)=\{y\}$ is residual in $Y$, thus demonstrating that $\pi$ is almost 1-1. It is clear from the construction that all non-singleton fibers $\pi^{-1}(x)$ are pseudo-arcs and that all pseudo-arcs are contained in the same composant $W.$ An almost 1-1 extension of a minimal system is always minimal, and so $H$ is minimal. The condition $\lim_{|i|\to \infty}\diam H^i\left(\pi^{-1}(x)\right)=0$ for all $x$ is a direct consequence of the construction. The proof is completed.
\end{proof}

\begin{thm}\label{nonminsq}
There exists a compact connected metric space $Y$ admitting a minimal homeomorphism such that $Y\times Y$ does not admit a minimal homeomorphism.
\end{thm}
\begin{proof}
\noindent
Let $h$ be the homeomorphism induced on a minimal Cantor set by a Denjoy extension of an irrational rotation of the unit circle and let $F$ be a minimal suspension of $h.$
Let $(Y,H)$ be a minimal dynamical system provided by an application of Theorem~\ref{thm:h}. By the construction of $Y$, all composants of $Y$ but one are one-to-one continuous images of $\R.$ There is also one special composant $W$ which consists of countably many pseudo-arcs connected by arcs.

We shall show that for any homeomorphism $G\colon Y\times Y\to Y\times Y$
there is an $n>0$ such that $G^n\colon Y\times Y\to Y\times Y$ is of the form $G^n(a,b)=(g_1(a),g_2(b))$, with $g_1,g_2\colon Y\to Y$ homeomorphisms. Since composants are dense in a continuum, it is enough to consider the composant $W\times W$.
\begin{figure}[h]
	\centering
		\includegraphics[width=0.60\textwidth]{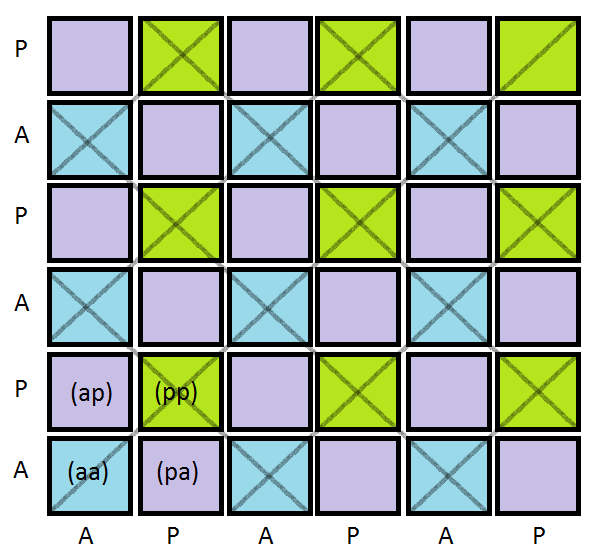}
	\label{fig:patiles}
	\caption{The structure of $W\times W$}
\end{figure}
This composant $W\times W$ is ``tiled" with the following 4 types of ``squares'': $(aa)$ $A\times A$, $(pp)$  $P\times P$, $(pa)$ $P\times A$ and $(ap)$ $A\times P$, where $P$ is a (maximal) pseudo-arc in $W$ and $A$ is an arc that connects two such pseudo-arcs. Note that any square of type $(ap)$ is homeomorphic with any square of type $(pa).$
\begin{claim}\label{pp}
No square of type $(pp)$ contains a homeomorphic copy of a square of type $(aa)$ or $(ap).$
\end{claim}
To prove the claim, suppose on the contrary that a square $\mathcal{T}$ of type $(pp)$ contains a copy of a square of type $(aa)$ or $(ap).$ Then $\mathcal{T}$ contains an arc $R$. Since the projection of $R$ onto one of the two coordinate spaces must be nondegenerate, it follows that a subcontinuum of $P$ is a continuous image of the arc. This contradicts the fact that $P$ does not contain any locally connected continuum, thereby establishing the claim.

\noindent
Applying similar arguments we also obtain the following.
\begin{claim}
No square of type $(ap)$ contains a homeomorphic copy of a square of type $(aa).$
\end{claim}

For simplicity of notation, by convention we will write $G(pp)=(pp)$ to denote that the image of a square of type $(pp)$ is contained in a square of type $(pp).$
It follows by the above claims that $G(pp)=(pp)$, $G(aa)=(aa)$, and $G(ap)=(ap)$ or $G(ap)=(pa)$. We call a point $x$ of a square of type $(pp)$ a \emph{corner point} if it also belongs to a square of type $(aa)$. By the above discussion, $G$ must transform corner points to corner points.
Therefore a diagonal in the lattice of corners of the form $$...(aa)(pp)(aa)(pp)(aa)...$$ must be mapped onto a diagonal of the same form,
because squares of type $(aa)$ may not ``cross'' squares of any other type. Then there is an $n\in \{1,2\}$ such that $G^n$ is a translation of diagonals, because if it translates one diagonal then, by the rigid structure of the diagonals
and corner points, it must also translate all other diagonals in exactly the same way. Denote by $E\times E$ the set of all corner points in $W\times W$ and observe that $E$ is dense in $Y$. By the above discussion there are homeomorphisms $g_1,g_2\colon E\to E$ such that $G^n|_{E\times E}=g_1\times g_2$. But then $g_1\times g_2$ can be extended continuously to $Y\times Y$ in a unique way, so $G^n=g_1\times g_2$, as claimed.

Now, by way of contradiction suppose $Y\times Y$ admits a minimal homeomorphism $G\colon Y\times Y$. By the above, we have $G^n=g_1\times g_2$ for some $n>0.$ Since $Y\times Y$ is a continuum, minimality of $G$ implies that $G^n$ is minimal as well. Then for simplicity we may assume that $n=1$. Note that since each $g_i$ is a homeomorphism, following the same lines of argument as above, $g_i$ must preserve the structure $\ldots (p)(a)(p)(a)(p)\ldots$ of the special composant $W$. Since the pseudo-arc has the fixed point property \cite{fpp}, the homeomorphism $g_i$ cannot permute a finite number of pseudo-arcs in $W$. This shows that $g_i$ must preserve the ordering of pseudo-arcs in $W$. More formally, if we enumerate consecutive pseudo-arcs in $W$, say $\{P_i\}_{i\in \Z}$ with $\pi(P_i)=F^i(q)$, then there are $k_1,k_2\neq 0$ such that $g_1(P_i)=P_{i+k_1}$ and $g_2(P_i)=P_{i+k_2}$. By our construction we see that $\lim_{|i|\to \infty}\diam P_i=0$ and
hence the relation $x\sim y$ iff $x=y$ or $x,y\in P_i$ for some $i$ is closed equivalence relation, and so $g_1$ induces a homeomorphism $f_1\colon X\to X$. But $f_1$ coincides with $F^{k_1}$ on a dense set
and so $f_1=F^{k_1}$. By the same argument $g_2$ is an extension of $F^{k^2}$. Since each suspension minimal system from a Denjoy homeomorphism is an extension of a rotation $R$ by an $\alpha\in \R\setminus \mathbb{Q}$ on the unit circle, we obtain that $G$
is an extension of the map $R^{k_1}\times R^{k_2}$ defined on the two dimensional torus. But the numbers $k_1\alpha$ and $k_2 \alpha$ are rationally dependent,
and hence $R^{k_1}\times R^{k_2}$ is not minimal (e.g. see \cite{Walters}). But any factor of a minimal homeomorphism has to be minimal, proving that $G$ is not minimal.
\end{proof}
\begin{rem}
It is clear from the proof that the Cartesian product of $n$ copies of $Y$, is not minimal for any $n>1$.
\end{rem}
The above construction, although carried out in an abstract setting, have a particularly nice realization, if we take into account that the suspension that we start with is a Denjoy minimal set for a torus homeomorphism \cite{Schweitzer}. Note that there exists uncountably many such nonhomeomorphic Denjoy minimal sets. 
\begin{thm}\label{surf}
There exists a compact connected metric space $Y$, and a torus homeomorphism $\varphi:\mathbb{T}^2\to \mathbb{T}^2$ homotopic to the identity, such that $Y$ is a minimal set of $\varphi$, but $Y\times Y$ does not admit a minimal homeomorphism.
\end{thm}
\begin{proof}
Let $\phi$ be a torus homeomorphism homotopic to the identity, with a minimal Denjoy continuum $X$. The proof is the same as in Theorem \ref{nonminsq}, with the exception that the proof of Theorem \ref{thm:h} needs a small adjustment (the rest is analogous), which we describe below. Let $g$ be defined as in the proof of Theorem \ref{thm:h}. We redefine each set $C_n$ putting $C_n:=[0-\frac{1}{n},0+\frac{1}{n}]$. Note that we still have $\bigcap C_n=\{0\}$. We put $U_n=[-1/n,1+1/n]\times C_n$ and use the Barge-Martin technique (see \cite{BMAttr}) to redefine each $f^{(1)}_n\colon\mathbb{T}^2\to\mathbb{T}^2$ to be a near-homeomorphism, such that $f_n|_{\mathbb{T}^2\setminus U_n}\equiv\operatorname{id}$ and $f^{(1)}_n$ is a homeomorphism everywhere, except on $[0,1]\times\{0\}$, where $f^{(1)}_n(x,0)=(g(x),0)$ and $g$ is the Henderson map \cite{Henderson}. Since each $f^{(1)}_n$ is near-homeomorphism, their inverse limit is homeomorphic to $\mathbb{T}^2$. This way we ``replace'' an arc in $\mathbb{T}^2$ by a pseudo-arc. All further steps of the construction are updated accordingly.
\end{proof}
\begin{que}
Does there exist a minimal space $Y$, such that $Y^m$ is minimal for some $m>1$, but $Y^k$ is not minimal for some $k>m$?
\end{que}

\section{Local product structure}

While we have only applied our technique to spaces supporting a flow, in fact the technique only relies on having a map that preserves a local product structure along with the ``angular'' relation. In the case of flows, this was accomplished by using the local translation structure, but this can also be achieved with certain classes of smooth maps.  A large class of skew products provide examples of such maps as we illustrate below. In addition to the intrinsic interest of the extension of the construction, it allows us to produce an example of a minimal but noninvertible map of the Klein bottle, which is known not to admit a minimal flow. This raises the possibility that the technique could lead to similar examples on a significantly larger class of manifolds than those that admit minimal flows.

\begin{thm}\label{thm:skewt2}
Let the homeomorphism $F\colon \mathbb{T}^2\to \mathbb{T}^2$ be a skew product defined by $F(x,y)=(x+\alpha, y+r(x)),$ where $r$ is a differentiable function on $S^1$.
For every $z\in \mathbb{T}^2$ there exists a map
 $f\colon \mathbb{T}^2\to \mathbb{T}^2$ and an almost 1-1 factor map $\pi \colon (\mathbb{T}^2,f)\to (\mathbb{T}^2,F)$
such that $\pi^{-1}(y)$ is a singleton for every $y\not \in \{F^n(z) : n\in \Z\}$ and which is an arc otherwise. Furthermore $\lim_{|n|\to\infty}\diam \pi^{-1}(F^n(z))=0$.
\end{thm}
\begin{proof}
Fix any $z=(x,y)\in \mathbb{T}^2$ and let $U$ be a small open neighborhood of $z.$ First observe that if $z_n=(x,y_n)\in U$ then $F(z_n)=(x+\alpha,y_n+r(x)),$
which means that $F$ preserves vertical lines in $U$. Now assume that $z_n$ converges to $z$ on a line which is not vertical; that is, $x_n\neq x$
and there is $\beta\in \R$ such that $\frac{y_n-y}{x_n-x}=\beta$ for every $n$. Then if we write a similar formula for the coordinates of $F(z_n)$ and $F(z)$ we obtain
$$
\frac{F(z_n)_2-F(z)_2}{F(z_n)_1-F(z)_1}=\frac{y_n +r(x_n)-y-r(x)}{x_n-x}=\beta + \frac{r(x_n)-r(x)}{x_n-x}\to \beta + r'(x).
$$
This implies that an arc obtained as the image of a radial line around $z$ in a small neighborhood has an asymptote at $F(z)$ and the asymptote is different for different $\beta$, and  lines of all possible directions as obtained as such an asymptote.
Then we can perform  the ``blow up'' procedure as in Theorem~\ref{generalnonivert} applied to the entire orbit of $z$, except for the compactification we use the circle of lines through the origin projected to an arc to compactify the pinched torus $\mathbb{T}^2\setminus\{z,F(z)\}$ with an additional arc. This makes the corresponding map on the resulting inverse limit space (which is still the torus) continuous because the assignment of a line to the corresponding asymptote is continuous.

By the construction using inverse limits, it is also clear that the condition $\lim_{|n|\to\infty}\diam \pi^{-1}(F^n(z))=0$ is satisfied.
\end{proof}

Define now the relation $\sim$ on $\mathbb{T}^2$ by $(x,y)\sim (x,y)$ and $(x,y)\sim (x+1/2,1-y)$. Clearly $\sim$ is a closed equivalence relation, and it is well known that $\mathbb{T}^2/\sim$ is the Klein bottle $\mathbb{K}$ and we use $p\colon \mathbb{T}^2 \to \mathbb{K}$ to denote the projection.

\begin{rem}
Let the homeomorphism $F\colon \mathbb{T}^2\to \mathbb{T}^2$ be a skew product defined by $F(x,y)=(x+\alpha,y+r(x))$ for some $\alpha \in \R$.
If $r(x+1/2)=1-r(x)$ then $F$ induces a homeomorphism $G$ on $\mathbb{K}$ satisfying $p\circ F = G\circ p.$ 
\end{rem}

\begin{thm}\label{thm:skewt3}
Let the homeomorphism $F\colon \mathbb{T}^2\to \mathbb{T}^2$ be a skew product defined by $F(x,y)=(x+\alpha, y+r(x))$ which preserves the relation $\sim$,
and let $G$ be the homeomorphism induced on the Klein bottle $\mathbb{K}$.
For every $z\in \mathbb{K}$ there exists a map $g\colon \mathbb{K}\to \mathbb{K}$ and an almost 1-1 factor map $\pi \colon (\mathbb{K},g)\to (\mathbb{K},G)$
such that $\pi^{-1}(y)$ is a singleton for every $y\not \in \{G^n(z) : n\in \Z\}$ and an arc otherwise. Furthermore $\lim_{|n|\to\infty}\diam \pi^{-1}(G^n(z))=0$.
\end{thm}
\begin{proof}
The proof is the same as in Theorem~\ref{thm:skewt2}. Simply, if we aim to ``blow up'' the  point in $z\in\mathbb{K}$ then, since $p$ is a local homeomorphism, we can pull $z$  back to $\{x,y\}=p^{-1}(z)$ in $\mathbb{T}^2$ and blow up the orbits of both $x$ and $y$ as in Theorem~\ref{thm:skewt2} and then project this back to $\mathbb{K}$ via the local homeomorphism $p.$ We can then extend $G$ to the compactifying intervals using the same limits (radial lines) as for $F$.
\end{proof}

The procedure described above works regardless of the special properties that $F$ might have. In \cite{Parry} Parry presented a simple and elegant argument describing how to obtain a minimal homeomorphism of the Klein bottle by defining a function $r$ satisfying $r(x)=-r(x+1/2)$ and such that the skew product on $\mathbb{T}^2$ given by $F(x,y)=(x+\alpha, y+r(x))$ is minimal. In fact, $r$ was defined by a convergent Fourier series of a special type, allowing Parry to obtain a family of minimal systems on $\mathbb{K}$ in this way.
Motivated by this, Sotola and Trofimchuk  showed by careful calculations in \cite{Sotola} formulas to modify Parry examples to  non-minimal systems. They also mentioned (see Remark~2.9 in \cite{Sotola}) that their construction may lead
to a minimal homeomorphism $G$ on $\mathbb{K}$ with a continuum $D\subset \mathbb{K}$ with $\lim_{n\to\infty }\diam G^n(D)=0,$ which by results of \cite{Kolyada} may serve as an alternative proof. We see now that, based on the examples of Parry, we can obtain similar examples in a direct way. Recall that any almost 1-1 extension of a minimal homeomorphism is minimal.

\begin{cor}
Let $F(x,y)=(x+\alpha, y+r(x))$ be a minimal homeomorphism of $\mathbb{T}^2$ constructed by Parry, which induces a minimal homeomorphism $G$ on $\mathbb{K}$, with $r$ is differentiable and satisfying $r(x)=-r(x+1/2)$.
Then ``blowing up'' orbit of any point $z\in \mathbb{K}$ by Theorem~\ref{thm:skewt3} we obtain a minimal homeomorphism $h$ on $\mathbb{K}$
which is almost 1-1 extension of $G$ and contains an arc $D$ such that $\lim_{n\to\infty}\diam h^n(D)=0$. 
\end{cor}

Clearly, when constructing inverse limit in Theorem~\ref{thm:skewt3} we do not have to blow up forward orbit of $z$, which allows us to construct noninvertible examples as described at the beginning of this section.

\begin{cor}
Let $F(x,y)=(x+\alpha, y+r(x))$ be a minimal homeomorphism of $\mathbb{T}^2$ constructed by Parry, which induces a minimal homeomorphism $G$ on $\mathbb{K}$, with $r$ is differentiable and satisfying $r(x)=-r(x+1/2)$.
Then ``blowing up'' the backward orbit of any point $z\in \mathbb{K}$ as in Theorem~\ref{thm:skewt3} we obtain a minimal but noninvertible map $h$ on $\mathbb{K}$ which is almost 1-1 extension of $G$.
\end{cor}

\section*{Acknowledgements}
The authors are grateful to Lubomir Snoha, Tomasz Downarowicz and Dariusz Tywoniuk for numerous discussions on properties of minimal dynamical systems and their extensions. In part, this work was supported by  NPU II project LQ1602 IT4Innovations excellence
in science, by  Grant IN-2013-045 from the Leverhulme Trust for an International Network and MSK grant 01211/2016/RRC ``Strengthening international cooperation in science, research and education'', which supported research visits of the authors. 

Research of P. Oprocha was supported by National Science Centre, Poland (NCN), grant no. 2015/17/B/ST1/01259, and J. Boro\'nski's work was supported by National Science Centre, Poland (NCN), grant no. 2015/19/D/ST1/01184.


\begin{thebibliography}{99}
\bibitem{Aarts}  Aarts, J. M.;  Oversteegen, L. \textit{The dynamics of the Sierpi\'nski curve}. Proc. Amer. Math. Soc. \textbf{120} (1994), 965--968.
%\bibitem{ABZ} Aranson, S. Kh.; Belitsky, G. R.; Zhuzhoma, E. V. Introduction to the qualitative theory of dynamical systems on surfaces. Translated from the Russian manuscript by H. H. McFaden. Translations of Mathematical Monographs, 153. American Mathematical Society, Providence, RI, 1996.

\bibitem{AK}  Auslander, J.;  Katznelson, Y.\textit{ Continuous maps of the circle without periodic points}, Israel J. Math. \textbf{32} (1979), no. 4, 375--381.

\bibitem{AY}  Auslander, J.;  Yorke, J.A. \textit{
Interval maps, factors of maps, and chaos}. T\^ohoku Math. J. (2) \textbf{32} (1980), no. 2, 177--188.
\bibitem{BMAttr} Barge, M.;  Martin, J. \textit{ The construction of global attractors}. Proc. Amer. Math. Soc.\textbf{ 110} (1990), no. 2, 523--525.
%\bibitem{Barge} Barge, Marcy; Bruin, Henk; \v Stimac, Sonja. The Ingram conjecture. Geom. Topol. 16 (2012), no. 4, 2481--2516.
\bibitem{Blokh}  Blokh, A.;  Oversteegen, L.; Tymchatyn,  E.D.  \textit{On minimal maps of 2-manifolds.} Ergodic Theory Dynam. Systems, \textbf{25} (2005), 41--57.
\bibitem{Brown}  Brown, M. \textit{Some applications of an approximation theorem for inverse limits}. Proc. Amer. Math. Soc. \textbf{11} (1960), 478--483.
\bibitem{Bruin} Bruin, H.; Kolyada, S.; Snoha, L'. {\it
Minimal nonhomogeneous continua.} Colloq. Math. \textbf{95} (2003), no. 1, 123--132.
\bibitem{Dev}  Devaney, R. L. \textit{An introduction to chaotic dynamical systems}. Studies in Nonlinearity. Westview Press, Boulder, CO, 2003.
\bibitem{Dirbak} Dirbák, M. \textit{Minimal extensions of flows with amenable acting groups.}
Israel J. Math. \textbf{207} (2015), no. 2, 581--615. 
\bibitem{Downarowicz}  Downarowicz, T.; Snoha, L.; Tywoniuk,   D. {\it Minimal Spaces with Cyclic Group of Homeomorphisms} J. Dyn. Diff. Equat. (2015). doi:10.1007/s10884-015-9433-2
\bibitem{Fayad}  Fayad, B. \textit{Topologically mixing and minimal but not ergodic, analytic transformation on $\mathbb{T}^5$},
Bol. Soc. Brasil. Mat. (N.S.), \textbf{31} (2000), 277--285.
\bibitem{fpp} Hamilton, O. H., {\it A fixed point theorem for pseudo-arcs and certain other metric continua.} Proc. Amer. Math. Soc. \textbf{2}, (1951), 173--174.
\bibitem{Handel} Handel, M. {\it A pathological area preserving $C^\infty$ diffeomorphism of the plane.} Proc.
Amer. Math. Soc., \textbf{86}  (1982), 163--168.
\bibitem{Henderson} Henderson, G. W., {\it  The pseudo-arc as an inverse limit with one binding map.} Duke Math. J. \textbf{31} 1964 421--425.
\bibitem{Ingram}  Ingram, W. T. {\it Inverse limits on $[0,1]$ using piecewise linear unimodal bonding maps.} Proc. Amer. Math. Soc. \textbf{128} (2000), 279--286.
\bibitem{Ingram2} Ingram, W. T. {\it Periodicity and indecomposability.} Proc. Amer. Math. Soc. \textbf{123} (1995), no. 6, 1907--1916.
\bibitem{Hurder} Hurder, S.; Rechtman, A. {\it The dynamics of generic Kuperberg flows.}
 Ast\'erisque, \textbf{377} (2016).
 \bibitem{Knes} Kneser, H. {\it Regul\"are Kurvenscharen auf den Ringfl\"achen.} (in German) Math. Ann. \textbf{91} (1924), 135--154.
\bibitem{Kolyada} Kolyada, S., Snoha, L'., Trofimchuk, S.: {\it Noninvertible minimal maps.} Fund. Math. \textbf{168}  (2001), 141--163.
\bibitem{KK} Kuperberg, K. {\it A smooth counterexample to the Seifert conjecture.}
Ann. of Math. (2) \textbf{140} (1994), no. 3, 723--732.
\bibitem{Lopez}  Lopez,  W. \textit{An example in the fixed point theory of polyhedra.} Bull. Amer. Math. Soc. \textbf{73} 1967 922--924.
\bibitem{Minc} Minc, P.; Transue, W. R. R. {\it  A transitive map on $[0,1]$ whose inverse limit is the pseudoarc.} Proc. Amer. Math. Soc. 111 (1991), no. 4, 1165--1170.
\bibitem{Markley}Markley,  N. G.,  \textit{The Poincar\'e–Bendixson theorem for the Klein bottle}, Trans. Amer. Math. Soc., \textbf{135} (1969), 159--165.
\bibitem{NS} Nemytskii, V.V. ;  Stepanov, V. V., \textit{  Qualitative theory of differential equations} Princeton Mathematical Series, No. 22 Princeton University Press, 1960.
\bibitem{Parry} Parry, W. \textit{A note on cocycles in ergodic theory.} Compositio Math., \textbf{28} (1974), 343--350.
\bibitem{Schweitzer} Schweitzer, P. A. {\it Counterexamples to the Seifert conjecture and opening closed leaves of foliations.}
Ann. of Math. (2) \textbf{100} (1974), 386--400.
\bibitem{Sotola} \v{S}otola, J.,  Trofimchuk, S., \textit{Construction of minimal non-invertible skew-product maps on 2-manifolds}. Proc. Amer. Math. Soc. \textbf{144} (2016), 723--732.
\bibitem{Tywoniuk} Tywoniuk, D. {\it Minimal non-invertible transformations of solenoids.}
Colloq. Math. \textbf{127} (2012), no. 2, 243--252.
\bibitem{Walters} Walters, P. {\it  An introduction to ergodic theory.} Graduate Texts in Mathematics, \textbf{79}. Springer-Verlag, New York-Berlin, 1982.
\end{thebibliography}
\end{document}